\RequirePackage{fix-cm}
\documentclass[11pt]{svjour3}       
\smartqed  
\usepackage{amsmath,amsfonts,amssymb,mathtools,bm}
\usepackage{graphicx,subfigure}
\usepackage{braket,dsfont,url,enumerate}

\usepackage{pdfpages} 
\usepackage{cmap}
\usepackage{hyperref}
\usepackage{color,dsfont,braket}
\usepackage{bbm}
\usepackage[headings]{fullpage}
\usepackage{fullpage}
\usepackage{fancyhdr}
\usepackage{times}
\usepackage{cmap}
\usepackage{pgf}
\usepackage{comment}

\def\R{{\mathbb R}}
\def\C{{\mathbb C}}

\def\d{{\mathrm d}}

\begin{document}

\title{A high-order exponential integrator for nonlinear parabolic equations with nonsmooth initial data}


\author{Buyang Li  and Shu Ma
}


\institute{Buyang Li and Shu Ma \at
              Department of Applied Mathematics, The Hong Kong Polytechnic University, Hong Kong.\\
\email{buyang.li@polyu.edu.hk and maisie.ma@connect.polyu.hk}\\
This work is partially supported by the Hong Kong Research Grants Council (GRF project No. 15300519) and an internal grant of the university (project code ZZKK).
}

\date{}

\maketitle

\begin{abstract}
A variable stepsize exponential multistep integrator, with contour integral approximation of the operator-valued exponential functions, is proposed for solving semilinear parabolic equations with nonsmooth initial data. 
By this approach, the exponential $k$-step method would have $k^{\rm th}$-order convergence in approximating a mild solution, possibly nonsmooth at the initial time. In consistency with the theoretical analysis, a numerical example shows that the method can achieve high-order convergence in the maximum norm for semilinear parabolic equations with discontinuous initial data.\\

\end{abstract}

\section{\bf Introduction}\label{Se:intr}

Let $A$ be the generator of a bounded analytic semigroup on a Banach space $X$, with domain $D(A)\subset X$, and consider the abstract semilinear initial-value problem
\begin{equation}
\label{ivp}
\left \{
\begin{aligned}
u'(t) - Au(t)&=f(t,u(t))  \quad\mbox{for}\,\,\, t\in(0,T] ,\\
u(0)&=u_0 , 
\end{aligned}
\right .
\end{equation}
where $u_0\in X$ and $f:[0,\infty)\times X\rightarrow X$ is a smooth (locally Lipschitz continuous) function. 
A function $u\in C([0,T];X)$ is called a {\it mild solution} of \eqref{ivp} if it satisfies the integral equation 
\begin{equation}
\label{mild}
u(t)=e^{tA}u_0+\int_0^t e^{(t-s)A}f(s,u(s))\d s ,\quad \forall\, t\in(0,T],  
\end{equation}
where $e^{tA}$ denotes the semigroup generated by the operator $A$.

In the linear case $f(t,u)\equiv f(t)$, time discretization of \eqref{ivp} by a $k^{\rm th}$-order Runge--Kutta method satisfies the following error estimate: 
\begin{equation}
\label{p-error}
\|u_n-u(t_n)\| \le C\tau^k t_n^{-k}  
\quad\mbox{for}\,\,\, u_0\in X ,
\end{equation}
where $\tau$ denotes the stepsize of time discretization; see \cite{Lubich-Ostermann-1996,Thomee2006}.  
In particular, for a nonsmooth initial value $u_0\in X$, the methods have $k^{\rm th}$-order accuracy when $t_n$ is not close to zero. 
This result also holds for implicit backward difference formulae (BDF), exponential integrators \cite{Hochbruck-Ostermann-2010,Koskela-Ostermann-2013}, and fractional-order evolution equations \cite[Remark 2.6]{Jin-Li-Zhou-2017}. 

However, such high-order convergence as \eqref{p-error} does not hold when the source function $f(t,u)$ is nonlinear with respect to $u$. A counter example constructed in \cite{Crouzeix-Thomee-1987} shows that a $k^{\rm th}$-order Runge--Kutta method normally has only first-order convergence for a general nonsmooth initial data $u_0\in X$, i.e.,  
\begin{equation}
\label{1-error}
C_1\tau t_n^{-1} \le \|u_n-u(t_n)\| \le C_2\tau t_n^{-1}  
\quad\mbox{for}\,\,\, u_0\in X .
\end{equation}
Similarly, semi-implicit Runge--Kutta methods also suffer from this barrier of convergence rate \cite{Ostermann-Thalhammer-2000}. 
For nonlinear problems with nonsmooth initial data, existing error estimates for exponential integrators also yield only first-order convergence (see \cite{Hochbruck-Ostermann-2005,Mukam-Tambue-2018,Ostermann-Thalhammer-2000})
\begin{equation}
\label{1-error-exp}
\|u_n-u(t_n)\| \le C\tau 
\quad\mbox{for}\,\,\, u_0\in X . 
\end{equation}
No method has been proved to have high-order convergence for semilinear parabolic equations with general nonsmooth initial data $u_0\in X$.

Of course, if the initial value is sufficiently smooth and satisfies certain compatibility conditions, e.g., $u_0\in D(A^k)$, then $O(\tau^k)$ convergence can be achieved uniformly for $t_n\in[0,T]$ for the nonlinear problem \eqref{ivp}. 
This has been proved for most time discretization methods, including Runge--Kutta methods \cite{Crouzeix-Thomee-1987},  
implicit A$(\alpha)$-stable multistep methods \cite{Lubich-1990}, implicit--explicit BDF methods \cite{Akrivis-2015,Akrivis-Lubich-2015}, splitting methods \cite{Einkemmer-Ostermann-2015,Einkemmer-Ostermann-2016,Hansen-Kramer-Ostermann-2012}
and several types of exponential integrators \cite{Calvo-Palencia-2006,Hochbruck-Ostermann-2005,Hochbruck-Ostermann-2010,Ostermann-Thalhammer-Wright-2006}. 
Extension to quasi-linear parabolic problems has also been done; see  
\cite{Gonzalez-Thalhammer-2007,Gonzalez-Thalhammer-2015,Gonzalez-Thalhammer-2016,Hochbruck-Ostermann-2011,Hochbruck-Ostermann-Schweitzer-2009,Lubich-Ostermann-1995}. 
The error estimates presented in these articles do not apply to nonsmooth initial data. 

Due to the presence of the factor $t_n^{-1}$, the first-order convergence of Runge--Kutta methods cannot be improved by using variable stepsizes. However, compared with other time discretization methods, exponential integrators were proved to have an error bound of $O(\tau)$ uniformly for $t_n\in [0,T]$ for nonsmooth initial data $u_0\in X$, without the factor $t_n^{-1}$ appearing in the error estimates for other methods (see \cite{Hochbruck-Ostermann-2005,Mukam-Tambue-2018,Ostermann-Thalhammer-2000}). This uniform convergence 
motivates us to consider the possibility of constructing high-order exponential integrators with variable stepsizes. 

In this paper, we propose a variable stepsize exponential $k$-step integrator for \eqref{ivp} with general nonsmooth initial data $u_0\in X$, by choosing 
\begin{align}\label{stepsize}
\tau_n=O((t_n/T)^\beta \tau) ,\quad\mbox{for some}\,\,\,\beta> 1-\frac1k,
\end{align}
where $\tau_n=t_n-t_{n-1}$ denotes the $n^{\rm th}$ stepsize in the partition $0=t_0<t_1<\dots<t_N=T$, and $\tau$ the maximal stepsize.  
For the convenience of implementation, we also integrate in the numerical method (and the error analysis) an algorithm for approximating the exponential integrator by using the contour integral techniques developed in \cite{Lopez-Fernandez-2010,LFPS-2006,Schadle-LF-Lubich-2006}. 

The proposed variable stepsize method, with contour integral approximation of the exponential integrator, can achieve $k^{\rm th}$-order accuracy in approximating a mild solution of \eqref{ivp}, i.e., 
\begin{align} \label{conv-exp-int-k}
\max_{1\le n\le N}
\|u_n-u(t_n)\| 
\le C\tau^k . 
\end{align}
This is the first high-order convergence result in approximating semilinear parabolic equations with nonsmooth initial data (without any regularity in addition to $u_0\in X$).  
In view of the result \eqref{1-error} for the Runge--Kutta methods, the convergence result \eqref{conv-exp-int-k} shows the superiority of the variable stepsize exponential integrator for problems with nonsmooth initial data. 
The approximation of exponential integrator would require $O(\ln(\tau^{-1}))$ parallel solutions of linear equations, and there are $N=O(\tau^{-1})$ time levels by using the stepsize in \eqref{stepsize}. Therefore, the total computational cost is $O(\tau^{-1}\ln(\tau^{-1}))$ for an accuracy of $O(\tau^k)$. 

For rigorous analysis without extra regularity assumptions on the solution, we assume that the nonlinear source function satisfies the following estimates: 
\begin{align}
&\|f(t,u)-f(t,v)\|\le C_{f,u,v}\|u-v\|\,\,\,\mbox{for}\,\,\, u,v\in X,\quad\mbox{(local Lipschitz continuity)}
\label{Lipschitz}\\[15pt] 
&\bigg\|\frac{\d^\ell}{\d t^\ell} f(t,u(t))\bigg\| \hspace{137pt} \mbox{(smoothness in $t$ and $u$)} \nonumber \\
&\le C_{f,u,\ell}\sum_{j=1}^\ell \!\!\sum_{\,\,\,\,\, m_1+\cdots +m_j\le\ell} \|\partial_t^{m_1} u(t)\| \,\|\partial_t^{m_2} u(t)\|\,\cdots\,\|\partial_t^{m_j} u(t)\| ,
\quad \ell=0,1,\dots , \label{df}
\end{align} 
where $\|\cdot\|$ denotes the norm of $X$, $C_{f,u,v}$ is a constant depending on $f$, $\|u\|$ and $\|v\|$; similarly, $C_{f,u,\ell}$ is a constant depending on $f$, $\|u(t)\|$, $\ell$, and the summation above extends over all possible positive integers $m_1,\dots,m_j$ satisfying $m_1+\dots+m_j\le \ell$ for a given $j$. 

Assumptions \eqref{Lipschitz}--\eqref{df} are naturally satisfied by a general smooth function $f:\R\rightarrow\R$ in a semilinear parabolic partial differential equation (PDE)
\begin{equation}\label{example-u-f}
\left \{
\begin{aligned}
\partial_tu(x,t) - \Delta u(x,t)  & = f(u(x,t))  && \mbox{for}\,\,\,(x,t)\in\Omega\times (0,T] ,\\ 
u(x,t) &=0 &&\mbox{for}\,\,\,(x,t)\in\partial\Omega\times (0,T],\\
u(x,0) &=u_0(x)&&\mbox{for}\,\,\,x\in\Omega .
\end{aligned}
\right .
\end{equation}
In this case, the Dirichlet Laplacian $\Delta$ generates a bounded analytic semigroup on $X=C_0(\overline\Omega)$, the space of continuous functions on $\overline\Omega$ which equal zero on the boundary $\partial\Omega$; see \cite{Ouhabaz1995}. 
Furthermore, the smooth function $f$ naturally extends to a function of $u\in C_0(\overline\Omega)$, satisfying 
$$
\|f(u)-f(v)\|_{C_0(\overline\Omega)}
\le C\max_{|s|\le \|u\|+\|v\|} |\partial_sf(s)|\,\|u-v\|_{C_0(\overline\Omega)} 
$$
and 
\begin{align*}
\frac{\d}{\d t}f(u(x,t))
&=
 \partial_uf(u) \partial_tu ,\\
\frac{\d^2}{\d t^2}f(u(x,t))
&=\partial_u^2f(u)(\partial_tu)^2+\partial_uf(u)\partial_{tt}u,\\
\frac{\d^3}{\d t^3}f(u(x,t))
&=\partial_u^3f(u)(\partial_tu)^3
+3\partial_u^2f(u)\partial_tu\partial_{tt}u
+\partial_uf(u)\partial_{ttt}u,\\
...&
\end{align*}
Obviously, all these time derivatives of $f(u(x,t))$ satisfy \eqref{df}. 
Hence, the semilinear parabolic PDE \eqref{example-u-f} with a general smooth function $f:\R\rightarrow\R$ is an example of the abstract problem \eqref{ivp} satisfying assumptions \eqref{Lipschitz}--\eqref{df}. 

Assumption \eqref{Lipschitz} is the same as the local Lipschitz continuity assumption used in \cite{Calvo-Palencia-2006} and \cite{Hochbruck-Ostermann-2005}. In \cite{Calvo-Palencia-2006}, authors proved high-order convergence with an addition assumption that the solution is in $C^k([0,T],X)$, which is not satisfied when the initial data is nonsmooth, i.e., $u_0\in X$ instead of $D(A^k)$. In \cite{Hochbruck-Ostermann-2005}, authors proved high-order convergence of exponential integrators with an additional assumption that $\partial_t^kf(u(t))$ is uniformly bounded for $t\in[0,T]$, which is also not satisfied when the initial data is nonsmooth. These additional assumptions in \cite{Calvo-Palencia-2006,Hochbruck-Ostermann-2005} are replaced by \eqref{df} in this paper, which is used to prove the weighted estimates 
$$
\|\partial_t^\ell u(t)\|\le Ct^{-\ell}\quad\mbox{for}\quad \ell=1,\dots,k,
$$ 
which allow the solution to be nonsmooth at $t=0$. These weighted estimates are used to prove high-order convergence of the exponential integrator in this paper. 

Both the regularity analysis and the error analysis in this paper can be similarly extended to semilinear parabolic equations with smoothly varying time-dependent operators. However, the extension to quasilinear parabolic equations with nonsmooth initial data is still not obvious.

\section{Numerical method}
 
We denote by $\widehat g(z):=\int_0^\infty e^{-zt}g(t)\d t$ the Laplace transform of a given function $g$. 
Then we let $g(t):=f(t,u(t))$ and take the Laplace transform of \eqref{mild} in time. This yields 
\begin{align}\label{Lapl-u}
\hat u(z) 
&= (z-A)^{-1}u_0 + (z-A)^{-1}\widehat g(z) .
\end{align}
Since $A$ generates a bounded analytic semigroup on $X$, there exists an angle $\phi\in(0,\frac{\pi}{2})$ such that the operator $(z-A)^{-1}$ is analytic with respect to $z$ in the sector 
$$
\Sigma_{\pi-\phi}:=\{z\in\C: |{\rm arg}(z)|<\pi-\phi\}. 
$$
In order to use the established contour integral techniques of \cite{Lopez-Fernandez-2010,LFPS-2006}, we take inverse Laplace transform of \eqref{Lapl-u} along the contour 
$$
\Gamma_\lambda=\{\lambda (1-\sin(\alpha+{\rm i} s)):s\in\R\} \subset \Sigma_{\pi-\phi} ,
$$
where $\alpha
=\frac{\pi}{4}-\frac{\phi}{2}$ and $\lambda$ is to be determined. 
Then we have
\begin{align*}
u(t) 
&= \frac{1}{2\pi {\rm i}} \int_{\Gamma_\lambda} e^{tz}(z-A)^{-1}(\widehat g(z)  + u_0) \d z.  
\end{align*}
Similarly, by considering $t_{n-1}$ as the initial time, the solution at $t=t_n$ can be written as 
\begin{align}\label{expr-ut}
u(t_n) 
&= \frac{1}{2\pi {\rm i}} \int_{\Gamma_{\lambda_n}} e^{\tau_nz}(z-A)^{-1}(\widehat {g_n}(z) + u(t_{n-1})) \d z ,
\end{align}
where $g_n(s)=f(t_{n-1}+s,u(t_{n-1}+s))$. 

In \cite[Theorem 1]{Lopez-Fernandez-2010} the authors proved that, by choosing 
the parameter 
\begin{align}\label{lambda}
\lambda_n
=\frac{2\pi d K(1-\theta)}{\tau_n a(\theta) }  ,
\end{align}
with 
\begin{align*}
d=\frac{\alpha}{2} 
,\quad \theta=1-\frac1K
\quad\mbox{and}\quad
a(\theta)={\rm arccosh}\bigg(\frac{1}{(1-\theta)\sin(\alpha)}\bigg) , 
\end{align*}
there are quadrature nodes and weights on the contour $\Gamma_{\lambda_n}$, 
$$
z_\ell=\lambda (1-\sin(\alpha+{\rm i} \ell h))
\,\,\,
\mbox{and} 
\,\,\,
w_\ell=\frac{\lambda h}{2\pi} \cos(\alpha+{\rm i} \ell h),
\,\,\, \ell=-K,\dots,K ,
\,\,\,
\mbox{with}\,\,\,h=\frac{a(\theta)}{K},
$$
such that \eqref{expr-ut} can be approximated by a quadrature 
\begin{align*}
u(t_n) 
&\approx \sum_{\ell=-K}^K 
w_\ell e^{\tau_n z_\ell}(z_\ell-A)^{-1} (\widehat {g_n}(z_\ell) + u(t_{n-1}))
\end{align*}
with an error of $O(e^{-K/C})$. 

Therefore, if $u^{(\tau)}=(u_n)_{n=0}^N$ denotes the numerical approximation of $(u(t_n))_{n=1}^N$, then we approximate the source function 
$f(t,u(t))$ by an extrapolation polynomial of degree $k-1$:
$$
f_n(t;u^{(\tau)})
=\sum_{j=1}^k L_j(t)f(t_{n-j},u_{n-j}) \quad\mbox{for $t\in(t_{n-1},t_n]$},
$$ 
where $L_j(t)$ is the unique polynomial of degree $k-1$ such that 
$$
L_j(t_{n-i})=\delta_{ij},\quad i=1,\dots,k. 
$$ 
For $n\ge k+1$ and given numerical solutions $u_{n-j}$, $j=1,\dots,k$, we denote 
$$
g_n(s;u^{(\tau)})=f_n(t_{n-1}+s;u^{(\tau)})
,\quad
L_{j,n}(s)=L_{j}(t_{n-1}+s), 
$$
and compute  
\begin{align}\label{method}
u_n
&= \sum_{\ell=-K}^K 
w_\ell e^{\tau_n z_\ell}(z_\ell-A)^{-1} (\widehat {g_n}(z_\ell;u^{(\tau)}) + u_{n-1}) \nonumber \\
&= \sum_{\ell=-K}^K 
w_\ell e^{\tau_n z_\ell}(z_\ell-A)^{-1}\Big(\sum_{j=1}^k\widehat {L_{j,n}}(z_\ell)f(t_{n-j},u_{n-j}) + u_{n-1}\Big) . 
\end{align}

The numerical solution at the starting $k$ steps can be computed by using the exponential Euler method 
\begin{align}\label{exp-Euler}
u_n
&= \sum_{\ell=-K}^K 
w_\ell e^{\tau_n z_\ell}(z_\ell-A)^{-1}\big(z_\ell^{-1}f(t_{n-1},u_{n-1}) + u_{n-1}\big) ,
\quad n=1,\dots,k. 
\end{align}
Since the stepsize choice in \eqref{stepsize} implies $\tau_n=O(\tau^{\frac{1}{1-\beta}})=O(\tau^k)$ for the starting $k$ steps, the exponential Euler scheme \eqref{method} can keep the errors of numerical solutions within $O(\tau^k)$ at the starting $k$ steps.

The main result of this paper is the following theorem.\medskip

\begin{theorem}\label{MainTHM}
Let $u_0\in X$ and assume that the nonlinear problem \eqref{ivp} has a mild solution $u\in C([0,T];X)$. Then there exist constants $\tau_0$ and $c_0$ such that for $\tau\le \tau_0$ and $K\ge \frac32 c_0\ln(\tau^{-1})$, the solutions $u_n$, $n=1,\dots,N$, given by \eqref{method}-\eqref{exp-Euler} with stepsize choice from \eqref{stepsize}, satisfies the following error estimate:
\begin{align}\label{error}
\max_{1\le n\le N} \|u_n-u(t_n)\|
\le C\tau^k + C\tau^{-1}e^{-K/c_0} .
\end{align}
\end{theorem}

\begin{remark}{\upshape
For $K\ge (k+1)c_0\ln(\tau^{-1})$ there holds $\tau^{-1}e^{-K/c_0}\le \tau^k$. Therefore, $O(\ln(1/\tau))$ quadrature nodes are needed to have an error of $O(\tau^k)$. 
}
\end{remark}

\begin{remark}{\upshape
Instead of choosing different $\lambda_n$ at different time steps, one can also divide the time interval $[t_{k+1},T]$ into $O(\log(\tau^{-1}))$ parts $[\Lambda^{-j-1}T,\Lambda^{-j}T]$, $j=0,\dots,J=O(\log(\tau^{-1}))$, with 
\begin{align}
\lambda_n 
=\frac{2\pi d K(1-\theta)}{\Lambda^\beta\tau_{K_j} a(\theta) } \quad
\mbox{being constant for $t_n\in [\Lambda^{-j-1}T,\Lambda^{-j}T]$},
\end{align}
and
\begin{align*}
d=\frac{\alpha}{2} 
,\quad \theta=1-\frac1K
\quad\mbox{and}\quad
a(\theta)={\rm arccosh}\bigg(\frac{\Lambda}{(1-\theta)\sin(\alpha)}\bigg) ,  
\end{align*}
where $\tau_{K_j}$ denotes the minimal stepsize for $t_n\in [\Lambda^{-j-1}T,\Lambda^{-j}T]$. This was used in many articles; see \cite{Lopez-Fernandez-2010,LF-Lubich-Schadle-2008,LFPS-2006,Schadle-LF-Lubich-2006}. By this method, at most $O(\log(\tau^{-1}))$ different contours are needed to have an error of $O(\tau^k)$. }
\end{remark}

\section{Proof of Theorem \ref{MainTHM}}

The proof consists of two parts. In section \ref{reg}, we prove the regularity of the solution $u$. By using the regularity result, we estimate errors of numerical solutions in section \ref{sec:error}.

\subsection{Regularity of solution}\label{reg}
It is well known that the solution of a linear parabolic equation has higher regularity at positive time and satisfies the estimate $\|\partial_t^\ell u(t)\|\le Ct^{-\ell}$, $\ell=0,1,\dots$, for a nonsmooth initial data $u_0\in X$. 
In this subsection, we prove that this is also true for the nonlinear problem \eqref{ivp} if the source function $f$ is smooth with respect to $t$ and $u$ in the sense of \eqref{df}. 
Since we have not found such a result in the literature for semilinear parabolic equations, we present the proof in the following lemma.  \medskip

\begin{lemma}
If $u\in C([0,T];X)$ is a mild solution of \eqref{ivp}, then 
$u\in C^k((0,T];X)$ and 
$$
\|\partial_t^\ell u(t)\|\le Ct^{-\ell},\quad \ell=0,1,\dots,k.
$$
\end{lemma}
\begin{proof}
If $u\in C([0,T];X)$ then the constant $C_{f,u,\ell}$ in \eqref{df} is bounded for $1\le \ell\le k$. We simply denote this constant by $C$. 
By mathematical induction, we assume that for $m=0,\dots,\ell-1$,
\begin{equation}\label{mathind}
\|\partial_t^mu(t)\|\le Ct^{-m},\quad t\in(0,T] . 
\end{equation}
Then \eqref{df} implies 
\begin{equation}\label{df2}
\bigg\|\frac{\d^m}{\d t^m} f(t,u(t))\bigg\| \le Ct^{-m},\quad t\in(0,T] ,
\quad\mbox{for}\,\,\, m=0,\dots,\ell-1 .
\end{equation}
In the following, we prove that \eqref{mathind} also holds for $m=\ell$. 

Multiplying \eqref{mild} by $t^\ell$ yields 
\begin{align}\label{tlut}
t^\ell u(t)
&=t^\ell e^{tA}u_0
+\int_0^t (t-s+s)^\ell e^{(t-s)A}f(s,u(s))\d s \nonumber \\
&=t^\ell e^{tA}u_0
+\sum_{j=0}^\ell \left(\begin{subarray}{c}
\displaystyle \ell\\ 
\displaystyle j
\end{subarray}\right)
\int_0^t(t-s)^j e^{(t-s)A}s^{\ell-j} f(s,u(s))\d s\nonumber \\
&=:t^\ell e^{tA}u_0
+\sum_{j=0}^\ell\left(\begin{subarray}{c}
\displaystyle \ell\\ 
\displaystyle j
\end{subarray}\right) w_{\ell,j}(t) ,
\end{align}
with 
\begin{align} \label{def-gj}
w_{\ell,j}(t)=\int_0^t g_j(t, s)\d s
\quad\mbox{and}\quad
g_j(t, s) = (t-s)^j e^{(t-s)A}s^{\ell-j} f(s,u(s)) .
\end{align} 
Note that 
\begin{align} \label{diff-integral-change}
\partial_t^j w_{\ell,j}(t)
=\partial_t^j  \int_{0}^{t} g_j(t, s)\, ds  
&= \partial_t^{j-1}  \partial_t \int_{0}^{t} g_j(t, s)\, ds \notag \\
&= \partial_t^{j-1}\int_{0}^{t} \partial_t  g_j(t, s) \, ds
+ \partial_t^{j-1} \bigl[g_j(t, s) |_{s=t} \bigr]  \notag \\
&= \partial_t^{j-2}\int_{0}^{t} \partial_t^2 g_j(t, s) 
+ \partial_t^{j-2} [\partial_t g_j(t, s) |_{s=t}] + \partial_t^{j-1} [g_j(t, s) |_{s=t}] \notag \\
&= \cdots \notag \\ 
&= \int_{0}^{t} \partial_t^j g_j(t, s) \, ds
+ \sum_{m=1}^j \partial_t^{j-m} [\partial_t^{m-1} g_j(t, s) |_{s=t}]  .
\end{align}
From the expression of $g_j$ in \eqref{def-gj} we know that $\partial_t^{m-1} g_j(t, s) |_{s=t}\equiv 0$ for $m=1,\dots,j$. As a result, we have  
\begin{align*} 
\partial_t^j w_{\ell,j}(t)   
&= \int_{0}^{t} \partial_t^j g_j(t, s) \, ds .\\
&= \int_0^t \partial_t^j[(t-s)^j e^{(t-s)A}]s^{\ell-j} f(s,u(s))\d s \\
&= \int_0^t \partial_s^j \bigl[s^j e^{sA}\bigr]  (t-s)^{\ell-j} f(t-s,u(t-s)) \, \d s 
\quad\mbox{(change of variable)} \\
&=: \int_0^t h_{\ell-j}(t, s) \, \d s . 
\end{align*}
Since the function $h_{\ell-j}(t, s) = \partial_s^j \bigl[s^j e^{sA}\bigr]  (t-s)^{\ell-j} f(t-s,u(t-s))$ contains a factor $(t-s)^{\ell-j}$, by a similar argument as \eqref{diff-integral-change} we have  
\begin{align*}
\partial_t^{\ell-j} \partial_t^j w_{\ell,j}(t)   
= \partial_t^{\ell-j} \int_{0}^{t}  h_{\ell-j}(t, s) \, ds 
= \int_{0}^{t} \partial_t^{\ell-j} h_{\ell-j}(t, s) \, ds ,
\end{align*}
which implies that 
\begin{align*}
\partial_t^{\ell} w_{\ell,j}(t) 
= \int_0^t \partial_s^j(s^j e^{sA})\frac{\d^{\ell-j}}{\d t^{\ell-j}}[(t-s)^{\ell-j} f(t-s,u(t-s))]\d s , 
\quad \mbox{for} \, \, 0 \le j \le \ell . 
\end{align*}
As a result, we have 
\begin{align}\label{dtw}
\|\partial_t^{\ell}w_{\ell,j}(t) \|
&\le \int_0^t C\|\partial_s^j(s^j e^{sA})\|_{X\rightarrow X}
\bigg\|\frac{\d^{\ell-j}}{\d t^{\ell-j}}[(t-s)^{\ell-j} f(t-s,u(t-s))]\bigg\|
\d s \nonumber \\
&\le \int_0^t C
\bigg\|\frac{\d^{\ell-j}}{\d t^{\ell-j}}[(t-s)^{\ell-j} f(t-s,u(t-s))]\bigg\|
\d s ,
\end{align}
where we have used 
$\|\partial_s^j(s^j e^{sA})\|_{X\rightarrow X}\le C,$ 
which is a consequence of the analytic semigroup estimate 
$$\|\partial_s^me^{sA} \|_{X\rightarrow X}\le Cs^{-m},\quad m=0,1,\dots $$  

If $1\le j\le \ell$ then substituting \eqref{df2} into \eqref{dtw} yields 
\begin{align}\label{dtwj}
\|\partial_t^{\ell}w_{\ell,j}(t) \|
\le C ,\quad 1\le j\le \ell. 
\end{align}

If $j=0$ then substituting \eqref{df} and \eqref{df2} into \eqref{dtw} yields 
\begin{align}\label{dtw0}
\|\partial_t^{\ell}w_{\ell,0}(t) \|
&\le \int_0^t C
\bigg\|\frac{\d^{\ell}}{\d t^{\ell}}[(t-s)^{\ell} f(t-s,u(t-s))]\bigg\|
\d s \nonumber \\
&\le \int_0^t C\sum_{j=1}^\ell 
\bigg\|[(t-s)^{\ell-j} \frac{\d^{\ell-j}}{\d t^{\ell-j}}f(t-s,u(t-s))]\bigg\|
\d s \nonumber \\
&\quad
+\int_0^t C
\bigg\|[(t-s)^{\ell} \frac{\d^{\ell}}{\d t^{\ell}}f(t-s,u(t-s))]\bigg\|
\d s 
\qquad\mbox{(product rule)} \nonumber \\
&\le C+\int_0^t C
\bigg\|[(t-s)^{\ell} \frac{\d^{\ell}}{\d t^{\ell}}f(t-s,u(t-s))]\bigg\|
\d s 
\end{align}
where we have used \eqref{df2} in estimating $(t-s)^{\ell-j}\frac{\d^{\ell-j}}{\d t^{\ell-j}}f(t-s,u(t-s))$ for $j\ge 1$.  
By considering the cases $j\ge 2$ and $j=1$ in \eqref{df}, separately, we have 
\begin{align*}
&\bigg\|\frac{\d^{\ell}}{\d t^{\ell}}f(t-s,u(t-s))\bigg\| \\ 
&\le 
C\sum_{j=2}^\ell \!\!\sum_{\,\,\,\,\, m_1+\cdots +m_j\le\ell} \|\partial_t^{m_1} u(t-s)\| \,\|\partial_t^{m_2} u(t-s)\|\,\cdots\,\|\partial_t^{m_j} u(t-s)\|
+ C
\|\partial_t^\ell u(t-s)\| \\ 
&\le    C(t-s)^{-\ell} 
+C
\|\partial_t^\ell u(t-s)\| .
\end{align*}
Substituting the inequality above into \eqref{dtw0}, we obtain 
\begin{align}\label{dtw0-2}
\|\partial_t^{\ell}w_{\ell,0}(t) \|
&\le C+\int_0^t C\|(t-s)^{\ell} \partial_t^\ell u(t-s)\| 
\d s = C+\int_0^t C\|s^{\ell} \partial_s^\ell u(s)\| 
\d s .
\end{align}
Then substituting \eqref{dtwj} and \eqref{dtw0-2} into \eqref{tlut} yields 
\begin{align}\label{tlut2}
\|\partial_t^{\ell}(t^\ell u(t))\|
&\le \| \partial_t^{\ell}(t^\ell e^{tA})u_0\| 
+\sum_{j=0}^\ell \left(\begin{subarray}{c}
\displaystyle \ell\\ 
\displaystyle j
\end{subarray}\right) \|\partial_t^{\ell}w_{\ell,j}(t) \| \nonumber \\
&\le 
C+\int_0^t C\|s^{\ell} \partial_t^\ell u(s)\| 
\d s . 
\end{align}
By using the product rule we can derive that 
$$
\|t^\ell \partial_t^{\ell}u(t)\|
\le \|\partial_t^{\ell}(t^\ell u(t))\|
+C\sum_{j=1}^\ell \|t^{\ell-j} \partial_t^{\ell-j}u(t)\| 
\le 
\|\partial_t^{\ell}(t^\ell u(t))\|
+C ,
$$
where we have used the induction assumption \eqref{mathind} in the last inequality. The above inequality and \eqref{tlut2} imply 
\begin{align}\label{tlut3}
\|t^\ell \partial_t^{\ell}u(t)\|
&\le 
C+\int_0^t C\|s^{\ell} \partial_t^\ell u(s)\| 
\d s . 
\end{align}
By using Gronwall's inequality, we derive 
\begin{align}\label{tlut4}
\|t^\ell \partial_t^{\ell}u(t)\|
&\le 
C ,\quad \forall\, t\in(0,T]. 
\end{align}
This proves \eqref{mathind} for $m=\ell$, and therefore the mathematical induction is closed. 
\hfill\end{proof}

\subsection{Error estimate}\label{sec:error}

We shall introduce a function $v(t)$ which is intermediate between $u(t_n)$ and $u_n$, and denote 
\begin{align}\label{def-en-eta-xi}
e_n:=u(t_n)-u_n,\quad 
\eta(t):=u(t)-v(t),
\quad\mbox{and}\quad
\xi_n:=v(t_n)-u_n ,
\end{align}
which imply the error decomposition 
$$
e_n=\eta(t_n)+\xi_n.
$$
Then we shall estimate $\eta(t_n)$ and $\xi_n$ separately. 

To this end, we define $v(t_k)=u_k$ and consider $n\ge k+1$: for given $v(t_{n-1})$ we define $v(t)$ for $t\in(t_{n-1},t_n]$ by   
\begin{align}\label{expr-vt}
v(t) 
&=  \frac{1}{2\pi {\rm i}} 
\int_{\Gamma_{\lambda_n}} e^{(t-t_{n-1}) z} (z-A)^{-1}
(\widehat g_n(z;u^{(\tau)})+v(t_{n-1}))\d z .
\end{align}
Comparing \eqref{expr-vt} with \eqref{expr-ut}, we see that $v(t)$ is actually the solution of the initial-value problem 
\begin{equation}\label{ivp-vv}
\left \{
\begin{aligned}
v'(t) - Av(t)&=f^{(\tau)}(t;u^{(\tau)})  \quad\mbox{for}\,\,\, t\in(t_k,T] ,\\
v(t_k)&=u_k ,
\end{aligned}
\right .
\end{equation}
where
$$
f^{(\tau)}(t;u^{(\tau)}) 
=f_n(t;u^{(\tau)}) ,\quad \mbox{for}\,\,\, t\in(t_{n-1},t_n], \quad n=k+1,k+2,\dots 
$$

To estimate $\eta(t_n)$, we consider the difference between \eqref{ivp} and \eqref{ivp-vv}. By using the notation in \eqref{def-en-eta-xi}, we see that $\eta(t)$ satisfies the following equation: 
\begin{equation}\label{ivp-u-w}
\left \{
\begin{aligned}
\eta'(t) - A\eta(t)&=f(t,u(t))-f^{(\tau)}(t,u^{(\tau)})  \quad\mbox{for}\,\,\, t\in(t_k,T] ,\\
\eta(t_k)&=e_k . 
\end{aligned}
\right .
\end{equation}
where
\begin{align*}
&\|f(t,u(t))-f^{(\tau)}(t,u^{(\tau)})\| \\ 
&=
\|f(t,u(t))-f^{(\tau)}(t;u(t_n)_{n=0}^N)+f^{(\tau)}(t;u(t_n)_{n=0}^N)-f^{(\tau)}(t,u^{(\tau)})\|\\
&\le C\tau_n^k \max_{t\in[t_{n-k},t_n]}
\big\|\mbox{$\frac{\d^k}{\d t^k}$}f(t,u(t))\big\|
+C_{n,u^{(\tau)}} \max_{1\le j\le k}\|e_{n-j}\| 
\quad\mbox{(use \eqref{Lipschitz}-\eqref{df} here)}\\
&\le C\tau_n^k t_{n-k}^{-k}+C_{n,u^{(\tau)}} \max_{1\le j\le k}\|e_{n-j}\| \\
&\le C\tau_n^k t_{n}^{-k}+C_{n,u^{(\tau)}} \max_{1\le j\le k}\|e_{n-j}\| , 
\quad \mbox{for}\,\,\, t\in(t_{n-1},t_n],\,\,\, n\ge k+1 .  
\end{align*}
where $C_{n,u^{(\tau)}} $ is a constant depending on $\|u_{n-j}\|$ for $j=1,\dots,k$. 

By using mathematical induction, we assume that 
\begin{align}\label{mathind-3}
\|u_{j}-u(t_j)\|=\|e_j\|\le 1
\quad\mbox{for}\,\,\, 
1\le j\le m-1 , 
\end{align}
then $C_{n,u^{(\tau)}} $ is bounded for $k+1\le n\le m$, and therefore 
\begin{align*}
&\|f(t,u(t))-f^{(\tau)}(t,u^{(\tau)})\| \\
&
\le C\tau_n^k t_{n}^{-k}+C \max_{1\le j\le k}\|e_{n-j}\| , 
\,\,\, \mbox{for}\,\,\, t\in(t_{n-1},t_n],\,\,\, k+1 \le n\le m .  
\end{align*}
Then we have 
\begin{align}\label{error-w}
\|\eta(t_{n})\|
&=
\bigg\|e^{(t_{n}-t_k)A}e_k+\int_{t_k}^{t_{n}} e^{(t_{n}-s)A}(f(s,u(s))-f^{(\tau)}(s;u^{(\tau)}))\d s\bigg\| \nonumber \\
&\le 
C\|e_k\|+
C\int_{t_k}^{t_{n}} \|f(s,u(s))-f^{(\tau)}(s;u^{(\tau)})\|\d s \nonumber \\ 
&\le 
C\|e_k\|+C\sum_{j=k+1}^n \tau_j \tau_j^k t_{j}^{-k}+C\sum_{j=1}^n\tau\|e_{j}\| \nonumber \\
&\le 
C\max_{0\le j\le k}\|e_j\|+C\sum_{j=k+1}^{n}\tau\|e_{j}\| 
+C\tau^k ,
\end{align} 
where we have used the following estimate in the last inequality: 
$$
\sum_{j=k+1}^n \tau_j \tau_j^k t_{j}^{-k}
\le 
C\tau^k\sum_{j=k+1}^n \tau_j  t_{j}^{k(\beta-1)} 
\le 
C\tau^k\int_{t_k}^{t_n} t^{k(\beta-1)} \d t
\le C\tau^k ,
\quad\mbox{if}\,\,\, k(\beta-1)+1>0. 
$$
This justifies the choice of $\beta$ in \eqref{stepsize}. 

To estimate $\xi_n=v(t_n)-u_n$, we note that 
\begin{align}\label{exp-vn}
v(t_n) 
&= \frac{1}{2\pi {\rm i}}
\int_{\Gamma_{\lambda_n}} e^{\tau_n z} (z-A)^{-1}
(\widehat g_n(z;u^{(\tau)})+v(t_{n-1}))\d z \nonumber \\
&= \frac{1}{2\pi {\rm i}}
\int_{\Gamma_{\lambda_n}} e^{\tau_n z} (z-A)^{-1}
(\widehat g_n(z;u^{(\tau)})+u_{n-1})\d z
+ e^{\tau_n A}\xi_{n-1} \nonumber \\
&= u_n
+ e^{\tau_n A}\xi_{n-1}\nonumber\\
&\quad 
+\frac{1}{2\pi {\rm i}}
\int_{\Gamma_{\lambda_n}} e^{\tau_n z} (z-A)^{-1}
(\widehat g_n(z;u^{(\tau)})+u_{n-1})\d z \nonumber \\
&\quad
- \sum_{\ell=-K}^K 
w_\ell e^{\tau_n z_\ell}(z_\ell-A)^{-1}(\widehat g_n(z_\ell;u^{(\tau)})+u_{n-1}) 
,
\end{align}
where we have used the identity $ \frac{1}{2\pi {\rm i}} \int_{\Gamma_{\lambda_n}} e^{\tau_n z} (z-A)^{-1}\xi_{n-1} \d z=e^{\tau_n A}\xi_{n-1}$. 
Since $$\widehat{g_n}(z;u^{(\tau)})=\sum_{j=1}^k\widehat {L_{j,n}}(z_\ell)f(t_{n-j},u_{n-j})$$
and the polynomial ${L_{j,n}}(z)$ satisfies $|\widehat {L_{j,n}}(z)|\le C|z|^{-\nu}$ for some $\nu\ge 0$, it follows that 
$$
\|(z-A)^{-1}
(\widehat g_n(z;u^{(\tau)})+u_{n-1})\|
\le C|z|^{-1}\Big(|z|^{-\nu} \sum_{j=1}^k\|f(t_{n-j},u_{n-j})\|+  \|u_{n-1}\|\Big). 
$$
For a function satisfying the estimate above, in \cite[Theorem 1]{LFPS-2006} (also see \cite{Lopez-Fernandez-2010}) the authors proved that 
\begin{align*}
&\bigg\|\frac{1}{2\pi {\rm i}}\int_{\Gamma_{\lambda_n}} e^{\tau_n z} (z-A)^{-1}
(\widehat g_n(z;u^{(\tau)})+u_{n-1})\d z \\
&\quad 
- \sum_{\ell=-K}^K 
w_\ell e^{\tau_n z_\ell}(z_\ell-A)^{-1}(\widehat g_n(z_\ell;u^{(\tau)})+u_{n-1}) \bigg\| \\
&
\le Ce^{-K/c_0}\Big(\sum_{j=1}^k\|f(t_{n-j},u_{n-j})\|+  \|u_{n-1}\|\Big)
\end{align*}
for some constant $c_0$. 
By choosing $e^{-K/c_0}=\tau^{k+1}$, which requires $m=O(\ln(\tau^{-1}))$, 
substituting the inequality above into \eqref{exp-vn} yields 
\begin{align*}
\|\xi_n-e^{\tau_n A}\xi_{n-1} \|
\le Ce^{-K/c_0}\Big(\sum_{j=1}^k\|f(t_{n-j},u_{n-j})\|+  \|u_{n-1}\|\Big) . 
\end{align*}

If \eqref{mathind-3} holds then $\|f(t_{n-j},u_{n-j})\|\le C$ for $j=1,\dots,k$ and $k+1\le n\le m$, and therefore 
\begin{align*}
&\|\xi_n-e^{\tau_n A}\xi_{n-1} \|\le Ce^{-K/c_0} \quad\mbox{for $k+1\le n\le m$}.
\end{align*}
Therefore, $q_n:=\xi_n-e^{\tau_nA}\xi_{n-1}$ satisfies $\|q_n\|\le Ce^{-K/c_0}$ and 
\begin{align*}
\xi_n
&=q_n+e^{\tau_nA}\xi_{n-1} \\
&=q_n+e^{\tau_nA}q_{n-1}+e^{(\tau_n+\tau_{n-1})A}\xi_{n-2} \\
&=q_n+e^{\tau_nA}q_{n-1}+e^{(\tau_n+\tau_{n-1})A}q_{n-2}+e^{(\tau_n+\tau_{n-1}+\tau_{n-2})A}\xi_{n-2} \\
&=\dots\\
&=\sum_{j=0}^{n-k-1} e^{(t_n-t_{n-j})A}q_{n-j} ,
\end{align*}
where the last equality holds because $\xi_k=0$. 
From the equation above we derive, for $k+1\le n\le m$, 
\begin{align}\label{error-xi}
\|\xi_n\|
\le 
\sum_{j=0}^{n-k-1} \|e^{(t_n-t_{n-j})A}q_{n-j}\|  
\le 
\sum_{j=0}^{n-k-1} C\|q_{n-j}\|  
\le 
\sum_{j=0}^{n-k-1} Ce^{-K/c_0}
\le C\tau^{-1}e^{-K/c_0}.
\end{align}

Combining \eqref{error-w} and \eqref{error-xi} and using the decomposition $e_n=\eta(t_n)+\xi_n$, we have 
\begin{align*} 
\|e_n\|
&\le \|\eta(t_{n})\|+\|\xi_n\| \\
&\le 
C\max_{0\le j\le k}\|e_j\|+C\sum_{j=k+1}^{n}\tau\|e_{j}\| 
+C\tau^k +C\tau^{-1}e^{-K/c_0},
\quad\mbox{for $k+1\le n\le m$}.
\end{align*} 
By using Gronwall's inequality, we obtain 
\begin{align} \label{alb}
\|e_n\|
&\le 
C\max_{0\le j\le k}\|e_j\|+C\tau^k+C\tau^{-1}e^{-K/c_0} , \quad\mbox{for $k+1\le n\le m$}.  
\end{align} 
If the starting steps are approximated sufficiently accurate, i.e., 
\begin{align} \label{k-step-assump}
\max_{0\le j\le k}\|e_j\|\le \frac12
\end{align} then there exists a positive constant $\tau_1$ such that for $m\ge \frac32 c_0\,\ln(1/\tau)$ (thus $e^{-K/c_0}\le \tau^{\frac32}$) and $\tau\le\tau_1$ there holds 
\begin{align} \label{mathind-cl}
\|e_m\|
&\le 
1.
\end{align} 
This completes the mathematical induction from \eqref{mathind-3} to \eqref{mathind-cl}, provided that \eqref{k-step-assump} holds. Then \eqref{alb} holds for $m=N$.

Since the starting $k$ steps are computed by the exponential Euler method, which is the special case $k=1$ of the analysis above. Therefore, the analysis above also implies 
\begin{align} \label{alb2}
\max_{0\le j\le k}\|e_j\| 
\le C\tau^k + Ce^{-K/c_0} . 
\end{align} 
This verifies \eqref{k-step-assump} for sufficiently small stepsize $\tau$ and sufficiently large $m$, say $\tau\le \tau_2$ and $m\ge m_2$. 
Then substituting \eqref{alb2} into \eqref{alb} yields 
\begin{align} \label{alb3}
\max_{k+1\le n\le N}\|e_n\|
&\le 
C\tau^k + C\tau^{-1}e^{-K/c_0} . 
\end{align} 

This completes the proof of Theorem \ref{MainTHM} under the stepsize condition $\tau\le \tau_0=\min(\tau_1,\tau_2)$ and $K\ge \frac32 c_0\ln(1/\tau)$.
\hfill\endproof

\section{Numerical example}

In this section, we present a numerical example to support our theoretical analysis and illustrate the convergence of the proposed time stepping method. 
Since the proposed numerical method is only for time discretization, which is independent of the spatial regularity of solution, we shall present a one-dimensional example with sufficiently accurate spatial discretization in order to observe the error and order of convergence of the time discretization method. 

We consider the nonlinear parabolic equation 
\begin{equation}
\label{ivp-Ex2}
\left \{
\begin{aligned}
\partial_tu(x,t) - \partial_{xx}u(x,t)&=u(x,t)-u^3(x,t) &&\mbox{for}\,\,\, (x,t)\in\Omega\times (0,T] ,\\
u(x,t)&=0 &&\mbox{for}\,\,\, (x,t)\in \partial\Omega\times (0,T], \\
u(x,0)&=u_0(x) &&\mbox{for}\,\,\, x\in \Omega,
\end{aligned}
\right .
\end{equation}
in a domain $\Omega \times (0,T)$ 
, with a discontinuous initial condition 
\begin{align}\label{ivp-Ex-u0}
u_0(x)=
\left\{\begin{aligned}
&0 &&x\in(0,0.5]\\
&1 &&x\in(0.5,1).  
\end{aligned}\right.
\end{align}
The function $f(u)=u-u^3$ is a smooth function of $u$ and therefore satisfying the assumptions \eqref{Lipschitz}--\eqref{df}, as mentioned in the example of semilinear parabolic equation \eqref{example-u-f}. 

The problem \eqref{ivp-Ex2} has a unique solution 
$$u\in C([0,T];L^p(\Omega))\cap C((0,T];C_0(\overline\Omega)),
\quad\mbox{with}\quad
u\notin C([0,T];L^\infty(\Omega)). $$
Therefore, $X=L^\infty(\Omega)$ does not fit the abstract problem directly. Nevertheless, the smoothing property of the heat semigroup guarantees that $u(\cdot,t)\in C_0(\overline\Omega)$ for arbitrarily small $t>0$ and therefore $X=C_0(\overline\Omega)$ would fit the abstract problem if we replace the initial time $t=0$ by an infinitesimal positive time. Therefore, Theorem \ref{MainTHM} implies that the numerical solution given by \eqref{method}-\eqref{exp-Euler} has an error bound of 
$$
\|u_n-u(t_n)\|_{C_0(\overline\Omega)} \le C\tau^k  
$$
for sufficiently large $K=O(\ln(1/\tau))$. 

We solve \eqref{ivp-Ex2} by the method \eqref{method}-\eqref{exp-Euler} with $\beta=\frac34$ for $k=2$ 
and $k=3$, respectively, using $\alpha= \frac{\pi}{4}$ and $K=10\,\log(1/\tau)$ quadrature nodes, and investigate the time discretization errors of the proposed time stepping method for several different $T$. 
The spatial discretization is done by using the standard finite difference method with a sufficiently small mesh size $2^{-10}$ so that further decreasing spatial mesh size has negligible influence in observing the order of convergence in time. 
The errors of numerical solutions between two consecutive stepszies are presented in Tables \ref{Table2} and \ref{Table3}, where the orders of convergence are computed by the formula
$$
\mbox{order of convergence} = 
\log\Bigg(\frac{\|u_N^{(\tau)}-u_N^{(\tau/2)}\|_{C_0(\overline\Omega)}}{\|u_N^{(\tau/2)}-u_N^{(\tau/4)}\|_{C_0(\overline\Omega)}}\Bigg)/\log(2) 
$$
based on the finest three meshes. 
The orders of convergence observed in these numerical tests are $O(\tau^k)$, which is consistent with the theoretical result proved in Theorem \ref{MainTHM}. 
 
%
%

\begin{table}[!htbp]
\centering
\caption{Numerical results of $\|u_N^{(\tau)}-u_N^{(\tau/2)}\|_{C_0(\overline\Omega)}$ for $k=2$.}
\begin{center}
\begin{tabular}{|l|c|c|c|c|c|}\hline 
& $T=1/2$ & $T=1/4$ & $T=1/8$ & $T=1/16$  \\ \hline
$\tau=$1/64    
&\!2.934$\times 10^{-6}$  
&\!3.935$\times 10^{-6}$ 
&\!1.286$\times 10^{-6}$
&\!\!1.864$\times 10^{-6}$ \\ \hline
$\tau=$1/128   
&\!7.410$\times 10^{-7}$ 
&\!9.351$\times 10^{-7}$ 
&\!3.053$\times 10^{-7}$ 
&\!\!7.297$\times 10^{-7}$ \\ \hline
$\tau=$1/256 
&\!1.779$\times 10^{-7}$ 
&\!2.269$\times 10^{-7}$ 
&\!7.735$\times 10^{-8}$ 
&\!\!1.838$\times 10^{-7}$\\ \hline
$\,$Order of convergence 
&$O(\tau^{2.1})$ & $O(\tau^{2.0})$ 
& $O(\tau^{2.0})$ 
& $O(\tau^{2.0})$  \\ \hline
\end{tabular}
\end{center}\bigskip
\label{Table2}
\centering
\caption{Numerical results of $\|u_N^{(\tau)}-u_N^{(\tau/2)}\|_{C_0(\overline\Omega)}$ for $k=3$.}\vspace{-10pt}
\begin{center}
\begin{tabular}{|l|c|c|c|c|c|}\hline 
& $T=1/2$ & $T=1/4$ & $T=1/8$ & $T=1/16$  \\ \hline
$\tau=$1/64    
&\!\!2.082$\times 10^{-7}$  
&\!\!5.807$\times 10^{-8}$\!
&\!\!2.945$\times 10^{-7}$
&\!\!3.425$\times 10^{-7}$ \\ \hline
$\tau=$1/128   
&\!\!2.614$\times 10^{-8}$ 
&\!\!7.756$\times 10^{-9}$\! 
&\!\!3.188$\times 10^{-8}$ 
&\!\!4.129$\times 10^{-8}$ \\ \hline
$\tau=$1/256 
&\!\!2.928$\times 10^{-9}$ 
&9.988$\times 10^{-10}$\!
&\!\!3.982$\times 10^{-9}$ 
&\!\!5.064$\times 10^{-9}$\\ \hline
$\,$Order of convergence 
&$O(\tau^{3.1})$ & $O(\tau^{3.0})$ 
& $O(\tau^{3.0})$ 
& $O(\tau^{3.0})$  \\ \hline
\end{tabular}
\end{center}
\label{Table3}
\end{table}

For comparison with the exponential integrator, we also present in Table \ref{Table_methods} the numerical results for the Crank--Nicolson method, 2-stage Gauss Runge--Kutta method and 2-stage Radau Runge--Kutta method for \eqref{ivp-Ex2}, 
with uniform stepsize $\tau=T/N$. The numerical results in Table \ref{Table_methods} show that the standard Crank--Nicolson method and Gauss Runge--Kutta method cannot yield any convergence rates. Indeed, these two methods do not satisfy the condition $|r(\infty)|<1$ in \cite[Theorem 1]{Crouzeix-Thomee-1987} when proving \eqref{example-u-f}. This shows the necessary of this condition in solving problems with nonsmooth initial data. The numerical results in Table \ref{Table_methods} also show that the 2-stage Radau Runge--Kutta method has roughly first-order convergence, instead of the optimal third-order convergence, for nonsmooth initial data. 

\begin{table}[!htbp]
\centering
\caption{Numerical results of $\|u_N^{(\tau)}-u_N^{(\tau/2)}\|_{C_0(\overline\Omega)}$ at $T=1/2$ (with $h = 2^{-14}$).}
\begin{center}
\begin{tabular}{|l|c|c|c|c|c|}\hline 
& $\tau=$1/256 & $\tau=$1/512 & $\tau=$1/1024 & Order of convergence \\\hline 
Crank--Nicolson    
&\!1.465$\times 10^{-1}$  
&\!1.466$\times 10^{-1}$
&\!1.466$\times 10^{-1}$ 
&$O(\tau^{0.0})$ \\ \hline
Gauss Runge--Kutta  
(2 stages)
&\!2.930$\times 10^{-1}$  
&\!2.930$\times 10^{-1}$ 
&\!2.933$\times 10^{-1}$ 
&$O(\tau^{0.0})$ \\ \hline
Radau Runge--Kutta  
(2 stages)
&\!2.215$\times 10^{-8}$  
&\!1.085$\times 10^{-8}$ 
&\!4.022$\times 10^{-9}$
&$O(\tau^{1.2})$\\ \hline
\end{tabular}
\end{center}\bigskip
\label{Table_methods}
\end{table}

\section{Conclusion}
We have proved that a variable stepsize exponential multistep integrator, with contour integral approximation of the operator-valued exponential functions, can produce high-order accurate numerical solutions for a semilinear parabolic equation with nonsmooth initial data (with no differentiability at all). 
The numerical example also supports this theoretical result. Both the regularity analysis and the error analysis in this paper can be similarly extended to semilinear parabolic equations with time-dependent coefficients. However, the extension to quasilinear parabolic equations with nonsmooth initial data is not trivial. 

The proposed method in this paper is essentially the multistep ETD with variable stepsize and contour integral approximation to the exponential operator. We have proved the first high-order convergence result in approximating semilinear parabolic equations with nonsmooth initial data (without any regularity in addition to $u_0\in X$).
For smooth initial data the exponential time differencing Runge--Kutta (ETD--RK) method would have the same complexity as the proposed multistep exponential integrator in this paper, both requiring to solve the equation for $O(\tau^{-1})$ time levels to achieve the accuracy of $O(\tau^k)$. However, since high-order accuracy of ETD--RK method has not been proved for nonsmooth initial data, the computational complexity of ETD--RK to achieve the accuracy of $O(\tau^k)$ is still unknown in this case. We believe the techniques of this paper may also be adapted to ETD--RK to yield high-order convergence for nonsmooth initial data.

\bibliographystyle{abbrv}
\bibliography{exp_integrator}

\end{document}